\documentclass[11pt]{article}
\usepackage[margin=1in]{geometry}
\usepackage{amsmath,amssymb,amsthm,mathtools}
\usepackage{enumitem}
\usepackage{hyperref}
\usepackage{comment}
\usepackage{xcolor}
\newtheorem{theorem}{Theorem}
\newtheorem{proposition}{Proposition}
\newtheorem{corollary}{Corollary}
\newtheorem{lemma}{Lemma}
\theoremstyle{remark}
\newtheorem{remark}{Remark}
\newtheorem{definition}{Definition}

\DeclareMathOperator{\In}{In}
\newcommand{\RT}{R_T}

\usepackage{tikz}
\usetikzlibrary{decorations.pathreplacing,calc,arrows.meta}
\usepackage{pgfplots}
\pgfplotsset{compat=1.18}

\title{Edge Subdivision and the Perron Eigenvalue of Tree Ricci Matrices}
\author{
Shuliang Bai\thanks{Beijing Yanqi Lake Institute of Mathematical Sciences and Applications, China. Email: \texttt{baishuliang@bimsa.cn}}
\and
Haoxuan Cheng\thanks{School of Mathematical Sciences, Fudan University, Shanghai 200433, China. Email: \texttt{hxcheng25@m.fudan.edu.cn}, \quad \textit{Corresponding author}}
\and
Bobo Hua\thanks{School of Mathematical Sciences, LMNS, Fudan University, Shanghai 200433, China. Email: \texttt{bobohua@fudan.edu.cn}}
}
\date{\today}

\begin{document}
\maketitle

\begin{abstract}
The Ricci matrix $R_T$ of a finite tree encodes its discrete Einstein metrics via the Perron eigenvector, with Lin-Lu-Yau's Ollivier Ricci curvature: $\kappa = -\lambda_{\max}(R_T)$. We show that edge subdivision, the natural operation of lengthening a tree, can decrease, preserve, or increase $\lambda_{\max}$. Compressing each branch into a scalar feedback function via the Schur complement reduces the spectral problem to a one-dimensional Chebyshev equation. We obtain an exact one-step trichotomy, a scalar transmission equation for arbitrary length, and the long-chain limit. Examples on double stars, including an asymmetric case where subdivision strictly increases $\lambda_{\max}$, illustrate the theory.

\end{abstract}

%\noindent\textbf{Keywords:} Ricci matrix; discrete Einstein metric;
%Lin--Lu--Yau curvature; edge subdivision; spectral radius; Schur complement;
%Chebyshev polynomials.

%\noindent\textbf{MSC 2020:} 05C50, 05C12, 53C21.

\section{Introduction}

Let $T=(V,E)$ be a finite tree. We \cite{BaiChengHua2026} introduced
the \emph{Ricci matrix} $R_T\in\mathbb{R}^{E\times E}$ of a tree, defined by
\[
(R_T)_{e,e'} =
\begin{cases}
-\bigl(\frac{1}{d_x}+\frac{1}{d_y}\bigr), & e=e'=\{x,y\},\\[4pt]
\frac{1}{d_z}, & e\neq e',\ e\cap e'=\{z\},\\[4pt]
0, & e\cap e'=\varnothing,
\end{cases}
\]
where $d_v$ denotes the degree of vertex $v$. A key property of this matrix
is that a discrete Einstein metric on a tree (i.e., edge weights with constant
Lin--Lu--Yau curvature \cite{LinLuYau2011}, in the sense of optimal-transport
Ricci curvature on graphs \cite{Ollivier2009}) corresponds exactly to the
Perron eigenvector of $R_T$, with the Einstein curvature given by
$\kappa=-\lambda_{\max}(R_T)$. Consequently, the sign and magnitude of
$\lambda_{\max}(R_T)$ entirely govern the corresponding curvature parameter,
and the study of its variation under elementary tree operations is of
significant geometric and dynamical interest.

In \cite{BaiChengHua2026}, we proved that $\lambda_{\max}(R_T)\le 0$
implies $T$ must be a caterpillar and established monotonicity
results for leaf attachment. Building on this foundation, Cheng \cite{Cheng2026} provided a complete classification of trees with
$\lambda_{\max}(R_T)\le 0$. In \cite{BCH2026leaf}, we investigated the asymptotic behavior of the sequence $\lambda_k = \lambda_{\max}(R_{T_k})$ obtained by repeatedly adding pendant edges at a fixed vertex $k$ times.
While these works largely focus on the operation of
\emph{leaf attachment}, the dual operation of \emph{edge subdivision}---
inserting a degree-$2$ vertex into an existing edge---has remained less
understood, even though it is the natural mechanism for ``lengthening'' a
tree. This paper addresses the following question:

\begin{center}
\emph{When an internal edge of a tree is subdivided, how does
$\lambda_{\max}(R_T)$ change? When is the change strictly negative, and
what governs its long-chain limit?}
\end{center}

Edge subdivision  is a classic graph operation with numerous applications in graph theory.  For the adjacency matrix $A(G)$, the classical Hoffman--Smith
theorem~\cite{HoffmanSmith1975} states that subdividing an edge on an
\emph{internal path} strictly decreases the spectral radius, with the single
exceptional family $\widetilde D_n$ on which it is preserved. (An internal path
is a path whose two endpoints have degree at least three and whose interior
vertices have degree two; subdividing a \emph{pendant} edge, by contrast,
typically increases the spectral radius.) Apart from this one exception, the
decrease is unconditional on the class of internal edges --- a monotonicity that
initiated Hoffman's program on limit points of graph spectra; see
\cite{G08,H72,V19,Z06} for related developments.

The geometric stakes are immediate: since $\kappa=-\lambda_{\max}(R_T)$ for a
discrete Einstein tree, every comparison of $\lambda_{\max}$ under subdivision
is a statement about how the Einstein curvature responds to lengthening an
internal edge. Our analysis shows this response is genuinely two-sided: unlike
the adjacency case, subdivision of an internal edge of $R_T$ can decrease,
preserve, or increase $\lambda_{\max}$ (Theorem~\ref{thm:one-step}). The
contrast with Hoffman--Smith is structural, not incidental, and the two
theories mirror each other at their critical cases: the adjacency spectral
radius decreases off the single exceptional family $\widetilde D_n$ (on which
it is constant), whereas for the Ricci matrix there is a critical family
$S_{3,3}$ on which $\lambda_{\max}\equiv0$ for \emph{every} subdivision length,
flanked on both sides by strict decrease and strict increase. The mechanism is
the degree-weighted off-diagonal entries $1/d_z$ of $R_T$, absent in $A(G)$: it
is precisely this degree dependence that breaks the one-sided monotonicity.

Degree-weighted adjacency matrices, where an edge $ij$ carries weight
$f(d_i,d_j)$, have been studied along similar lines; Shen and Shan
\cite{ShenShan2024} analyse the effect of subdividing an edge on the
$f$-spectral radius and find, in the same spirit, that the comparison depends on
the chosen edge rather than holding unconditionally, building on the interlacing
results of Li and Yang \cite{LiYang2023}. Our setting differs in both the matrix
and the conclusion: the Ricci coupling is genuinely two-sided in the degrees
(diagonal $-(1/d_x+1/d_y)$, off-diagonal $1/d_z$) rather than a single positive
weight $f(d_i,d_j)$, and we obtain an \emph{exact} trichotomy criterion
(Theorem~\ref{thm:one-step}) rather than one-directional sufficient conditions.

%This question on the Ricci matrix is more delicate:
%because its off-diagonal entries $1/d_z$ depend on the degree of the shared
%vertex, the effect of inserting a degree-$2$ vertex is sensitive to the
%\emph{balance} between the two sides of the subdivided edge. Subdivision
%can decrease, preserve, or even increase $\lambda_{\max}(R_T)$, depending
%on the global structure of the tree.
\medskip
We provide a systematic interface-based analysis of edge
subdivision on discrete Einstein tree $T$. Fix an internal edge $xy\in E(T)$, and for $\ell\ge 1$ let
$T^{(\ell)}$ denote the tree obtained by replacing $xy$ with a path of
length $\ell$. Deleting $xy$ splits $T$ into two rooted branches $(A,x)$
and $(B,y)$. By Schur eliminating the old-edge blocks, we compress the
two branches into scalar \emph{feedback functions} $\Phi_A(\lambda)$
and $\Phi_B(\lambda)$, and the entire spectral problem of $R_{T^{(\ell)}}$
reduces to a one-dimensional interface equation $D_\ell(\lambda)=0$
along the inserted degree-$2$ chain. From this reduction, we obtain
the following results.

\begin{itemize}
\item \textbf{Exact one-step criterion}
(Theorem~\ref{thm:one-step}).
Let $\lambda_1=\lambda_{\max}(R_T)$ and
$\lambda_2=\lambda_{\max}(R_{T^{(2)}})$. Then
\[
  \lambda_2<\lambda_1
  \;\Longleftrightarrow\;
  uv>-\frac{\lambda_1}{2},
\]
where $u,v$ are interface quantities determined by the Perron vector of
$R_T$ (see Definition~\ref{def:uv}).

\item \textbf{Degree-dependent sufficient condition}
(Corollary~\ref{cor:degree}). If $d_x\ge 3$, $d_y\ge 3$, and
\[
  \lambda_1>
  \max\!\left\{\frac{2}{d_x(d_x-2)},\frac{2}{d_y(d_y-2)}\right\},
\]
then $\lambda_2<\lambda_1$.

\item \textbf{Scalar reduction for arbitrary length}
(Proposition~\ref{prop:general}). For every $\ell\ge 2$,
\[
  \lambda\in\sigma(R_{T^{(\ell)}})
  \;\Longleftrightarrow\;
  D_\ell(\lambda)
  =
  -\Phi_A\Phi_B\,U_{\ell-2}(1+\lambda)
  +(\Phi_A+\Phi_B)\,U_{\ell-3}(1+\lambda)
  -U_{\ell-4}(1+\lambda)=0,
\]
where $U_n$ denotes the second-kind Chebyshev polynomial, with the
conventions $U_{-1}=0$ and $U_{-2}=-1$.

\item \textbf{Long-chain limit equation}
(Corollary~\ref{cor:limit}). If
$\lambda_\ell:=\lambda_{\max}(R_{T^{(\ell)}})$ eventually lies in a
compact positive interval avoiding the poles of $\Phi_A,\Phi_B$, then
every accumulation point $\lambda_\ast$ of $\{\lambda_\ell\}$ satisfies
$\Phi_A(\lambda_\ast)=q(\lambda_\ast)$ or $\Phi_B(\lambda_\ast)=q(\lambda_\ast)$,
where $q(\lambda)=1+\lambda-\sqrt{\lambda(\lambda+2)}$ is the decaying
root of the chain recurrence.
\end{itemize}

As an application, we analyze the symmetric double-star family
$T_{s,\ell}$ obtained by subdividing the central edge of $S_{s,s}$.
For this family the interface equation reduces explicitly to
$R_\ell(\lambda)=\beta_s(\lambda)$, yielding the threshold phenomenon at
$s=3$, strict monotonicity $\lambda_{s,\ell+1}<\lambda_{s,\ell}$ for $s>3$
with positive limit, and the explicit value
$\lambda_{s,2}=(\sqrt{s+1}-2)/(s+1)$ together with the asymptotics
$\lambda_{s,\ell}\sim\sqrt{2/(\ell(s+1))}$ and
$\lambda_{s,\infty}\sim 2^{1/3}(s+1)^{-2/3}$ as $s\to\infty$.

In addition, we present a numerical illustration for an asymmetric 
double star ($d_x=3$, $d_y=6$), which demonstrates that subdivision can 
\emph{increase} $\lambda_{\max}$ and that the long-chain limit is governed 
by the branch with the larger leaf count.

\medskip

The Schur complement is a standard tool for reducing the dimension of a spectral problem when a graph admits a cutset whose removal separates the graph into disjoint components. 

This technique has been successfully applied to study edge subdivision for adjacency matrices~\cite{hoffman_smith,kumar2025subdivision} and Laplacian matrices~\cite{chen2018,biane_chapuy}. In the present work, we adapt it to the Ricci matrix $R_T$ on trees. Unlike the adjacency case---where subdivision strictly decreases the spectral radius---we find that subdivision can increase, decrease, or preserve the largest eigenvalue, depending on the balance of the two branches. This non-monotonicity reflects the nonlinear coupling of weights and degrees in discrete Einstein metrics.
\medskip
The paper is organized as follows. Section~\ref{sec:notation} collects the
matrix-analysis preliminaries and introduces the Ricci matrix and the branch
feedback functions. Section~\ref{sec:one-step} treats the one-step case and
proves the exact comparison criterion together with the Schur-complement
sufficient condition. Section~\ref{sec:general} returns to arbitrary
subdivision length and develops the scalar transmission equation and the
long-chain limit criterion. Section~\ref{sec:ss} first analyzes the symmetric double-star family in detail, 
and then presents a numerical illustration for an asymmetric double star.

\section{Ricci Matrices and Branch Feedback Functions}\label{sec:notation}

\subsection{The Ricci matrix and discrete Einstein trees}

Let $T=(V,E)$ be a finite tree. For a vertex $v$, let $d_v$ be its degree.
We call an edge $xy\in E(T)$ \emph{internal} if both endpoints have degree at
least two, i.e.\ $d_x\ge2$ and $d_y\ge2$; equivalently, neither endpoint is a
leaf. For an edge $e=\{x,y\}$, the \emph{Lin--Lu--Yau Ricci curvature} of a
weighted tree with edge weights $w$ is \cite{BaiChengHua2026}
\[
\kappa_{xy} = -\left( \frac{S_x-2w_{xy}}{w_{xy}d_x} + \frac{S_y-2w_{xy}}{w_{xy}d_y} \right),
\qquad S_v = \sum_{u\sim v} w_{uv}.
\]
Its \emph{Ricci matrix} $\RT$ is
\[
(\RT)_{e,e'}
=
\begin{cases}
-\left(\dfrac1{d_x}+\dfrac1{d_y}\right), & e=e'=\{x,y\},\\[4pt]
\dfrac1{d_z}, & e\neq e',\ e\cap e'=\{z\},\\[4pt]
0, & e\cap e'=\varnothing.
\end{cases}
\]
A tree is \emph{discrete Einstein} if there is a positive weight function $w$
with $\kappa_{xy}\equiv\kappa$ constant. We \cite{BaiChengHua2026} proved that
this is equivalent to $\RT w=\lambda w$ with $\lambda=-\kappa$. By
Perron--Frobenius (Lemma~\ref{lem:pf} below), $\lambda_{\max}(\RT)$ is simple
with a strictly positive eigenvector (the Perron vector); hence
$\kappa=-\lambda_{\max}(\RT)$.

\subsection{Preliminaries}\label{sec:prelim}

We collect the matrix-analysis facts used throughout. All matrices below are
real symmetric. For such $M$ we write $M\succ0$ ($M\succeq0$) for positive
(semi)definiteness, $\sigma(M)$ for its spectrum, $\lambda_{\max}(M)$ for its
largest eigenvalue, and $\In(M)=(n_+,n_-,n_0)$ for its \emph{inertia}, the
numbers of positive, negative and zero eigenvalues counted with multiplicity.

\subsubsection{Schur complements}

Let $M=\begin{psmallmatrix} P & Q\\ Q^{T} & S\end{psmallmatrix}$ be symmetric
with the block $S$ invertible. The \emph{Schur complement of $S$ in $M$} is
$M/S:=P-QS^{-1}Q^{T}$. With $L=\begin{psmallmatrix} I & QS^{-1}\\ 0 & I\end{psmallmatrix}$
one has the congruence
\begin{equation}\label{eq:schur-congruence}
  M=L\begin{pmatrix} M/S & 0\\ 0 & S\end{pmatrix}L^{T},
\end{equation}
from which the following standard facts are immediate; see \cite{Zhang2005}.

\begin{lemma}[Schur complement]\label{lem:schur}
With the notation above:
\begin{enumerate}[label=\rm(S\arabic*)]
\item \textbf{(Determinant)} $\det M=\det S\cdot\det(M/S)$; in particular $M$
is singular iff $M/S$ is singular.
\item \textbf{(Inertia additivity; Haynsworth \cite{Haynsworth1968})}
$\In(M)=\In(S)+\In(M/S)$. In particular $M\succ0\iff S\succ0$ and $M/S\succ0$.
\item \textbf{(Elimination)} If $M\begin{psmallmatrix}\xi\\ \eta\end{psmallmatrix}=0$,
then $\eta=-S^{-1}Q^{T}\xi$ and $\xi$ solves $(M/S)\xi=0$.
\end{enumerate}
\end{lemma}

\begin{proof}
Taking determinants in \eqref{eq:schur-congruence} and using $\det L=1$ gives
(S1). Since congruent symmetric matrices have equal inertia (Sylvester) and the
block-diagonal matrix in \eqref{eq:schur-congruence} has inertia
$\In(M/S)+\In(S)$, we obtain (S2). For (S3), the lower block of
$M\begin{psmallmatrix}\xi\\ \eta\end{psmallmatrix}=0$ reads $Q^{T}\xi+S\eta=0$,
so $\eta=-S^{-1}Q^{T}\xi$; substituting into the upper block gives
$(M/S)\xi=0$.
\end{proof}

Two uses recur. \emph{Spectral comparison:} fixing a principal block of
$\lambda I-R$ and forming its Schur complement reduces questions about
$\sigma(R)$ to a smaller matrix---by (S1), $\lambda\in\sigma(R)$ iff that Schur
complement is singular, and by (S2), deciding $\lambda\gtrless\lambda_{\max}(R)$
becomes a definiteness check. \emph{Branch elimination:} applying (S3) to the
eigen-equation $(\lambda I-R)\mathbf w=0$ removes branch coordinates, producing
the feedback functions of Section~\ref{sec:feedback}.

\subsubsection{Perron--Frobenius structure}

\begin{lemma}[Perron--Frobenius preparation]\label{lem:pf}
Let $T$ be a tree with $|E(T)|\ge 2$.
\begin{enumerate}[label=\rm(\arabic*)]
\item $\RT+2I$ is entrywise nonnegative and irreducible. Consequently
$\lambda_{\max}(\RT)$ is a simple eigenvalue admitting a strictly positive
eigenvector, and it is the unique eigenvalue of $\RT$ with a positive
eigenvector.
\item If $N=\RT|_{F}$ is the principal submatrix on a proper subset
$F\subsetneq E(T)$, then $\lambda_{\max}(N)<\lambda_{\max}(\RT)$.
\end{enumerate}
\end{lemma}

\begin{proof}
(1) Each diagonal entry of $\RT$ equals $-(1/d_x+1/d_y)\ge-2$, so $\RT+2I$ has
nonnegative diagonal; its off-diagonal entries are $1/d_z\ge0$. Two edges have a
nonzero off-diagonal entry exactly when they share a vertex, i.e.\ are adjacent
in the line graph $L(T)$; since $T$ is connected, $L(T)$ is connected, so
$\RT+2I$ is irreducible. Perron--Frobenius then yields a simple largest
eigenvalue with a strictly positive eigenvector, unique among eigenvalues with a
positive eigenvector; the shift by $-2I$ leaves eigenvectors unchanged.
(2) $N+2I$ is a proper principal submatrix of the irreducible nonnegative matrix
$\RT+2I$, so $\lambda_{\max}(N+2I)<\lambda_{\max}(\RT+2I)$ by strict
Perron--Frobenius monotonicity; subtract $2$.
\end{proof}

\begin{lemma}[Eigenvalue monotonicity; Weyl]\label{lem:weyl}
If $A\preceq B$ are real symmetric, then $\lambda_{\max}(A)\le\lambda_{\max}(B)$;
in particular, adding a positive semidefinite matrix does not decrease
$\lambda_{\max}$ \cite[Cor.~4.3.12]{HornJohnson2013}.
\end{lemma}

\begin{lemma}\label{lem:mu-gt--1}
For any tree $T$ with at least three edges, $\lambda_{\max}(R_T) >  -1$.
\end{lemma}

\begin{proof}
Let $\Delta$ be the maximum degree of $T$. If $\Delta\ge 3$, then
Proposition~4 of \cite{BaiChengHua2026} gives
\[
  \lambda_{\max}(R_T)\ge -\frac2\Delta>-1.
\]
It remains to treat $\Delta=2$, i.e.\ $T$ is a path. Since $|E(T)|\ge 3$,
let $M$ denote the following $3\times 3$ principal block of $R_T$:
\[
  M=
  \begin{pmatrix}
    -\frac32 & \frac12 & 0\\[2pt]
    \frac12 & -1 & \frac12\\[2pt]
    0 & \frac12 & -\frac32
  \end{pmatrix}.
\]
If $\eta=(1,1,1)^T$ and $\widetilde\eta\in\mathbb R^{E(T)}$ is obtained by placing
the entries of $\eta$ on these three coordinates and zeros elsewhere, then
$\widetilde\eta^{T}R_T\widetilde\eta=\eta^{T}M\eta$ and
$\widetilde\eta^{T}\widetilde\eta=\eta^{T}\eta$. Hence, by the Rayleigh quotient,
\[
  \frac{\eta^TR_T\eta}{\eta^T\eta}\ge
  \frac{-\frac32-1-\frac32+2\cdot\frac12+2\cdot\frac12}{3}
  =-\frac23>-1,
\]
so $\lambda_{\max}(R_T)>-1$.
\end{proof}

\section{One-Step Subdivision}\label{sec:one-step}
%==========================================================================

We first treat the one-step case $\ell=1\to\ell=2$ by a direct argument
that uses no auxiliary functions and involves only a $2\times2$ matrix.

%\subsection{Setup}\label{sec:one-step-setup}

Fix an internal edge $xy\in E(T)$ with $d_x\ge2$, $d_y\ge2$.
Write $\lambda_1:=\lambda_{\max}(R_T)$. The \emph{subdivided tree}
$T^{(2)}$ is obtained by inserting a new vertex $v_1$ of degree~$2$
into $xy$, replacing it by two edges $e_1=xv_1$ and $e_2=v_1y$:
\[
  T:\quad
  \underbrace{\cdots}_{\text{branch }A}
  \text{---}\, x \,\text{---}\, y \,\text{---}
  \underbrace{\cdots}_{\text{branch }B}
  \qquad\longrightarrow\qquad
  T^{(2)}:\quad
  \underbrace{\cdots}_{\text{branch }A}
  \text{---}\, x \,\text{---}\, v_1 \,\text{---}\, y \,\text{---}
  \underbrace{\cdots}_{\text{branch }B}
\]
The branches $A$ (containing $x$) and $B$ (containing $y$) are unchanged.
Write $\lambda_2:=\lambda_{\max}(R_{T^{(2)}})$. The degrees of $x$ and $y$
are the same in both trees (each still has $d_x$ and $d_y$ incident edges
respectively); only the new vertex $v_1$ has degree~$2$.

Let $E_0=E(T)\setminus\{xy\}$ and let $R_0$ be the principal submatrix of
$R_T$ on $E_0$. Since the subdivision changes no degrees of vertices
incident to edges in $E_0$, the matrix $R_0$ is also the principal
submatrix of $R_{T^{(2)}}$ on $E_0$.

\subsection{The $2\times2$ Schur complement at $\lambda_1$}

\begin{lemma}\label{lem:lambda1-above-R0}
Let $\lambda_1 = \lambda_{\max}(R_T)$. Then $\lambda_1 > \lambda_{\max}(R_0)$; 
in particular $\lambda_1 I - R_0$ is positive definite.
\end{lemma}

\begin{proof}
$R_0+2I$ is a proper principal submatrix of the irreducible nonnegative
matrix $R_T+2I$, so $\lambda_{\max}(R_0)<\lambda_1$ by
Lemma~\ref{lem:pf}(2).
\end{proof}

We now compute the Schur complement of $\lambda_1 I-R_0$ inside
$\lambda_1 I-R_{T^{(2)}}$. Write $G:=(\lambda_1 I-R_0)^{-1}$ (positive
definite by Lemma~\ref{lem:lambda1-above-R0}). Define indicator vectors
$p,q\in\mathbb R^{E_0}$:
\[
  p_f=\begin{cases}1,&f\ni x,\\0,&\text{otherwise},\end{cases}\qquad
  q_f=\begin{cases}1,&f\ni y,\\0,&\text{otherwise},\end{cases}
\]
and set $s_x:=p^TGp/d_x^2$, $s_y:=q^TGq/d_y^2$, $g:=p^TGq/(d_xd_y)$.

The Schur complement is the $2\times2$ matrix
\begin{equation}\label{eq:S-def}
  S(\lambda_1)
  =\begin{pmatrix} a & -h\\ -h & b\end{pmatrix},
\end{equation}
where $a=\lambda_1+\tfrac1{d_x}+\tfrac12-s_x$,
$b=\lambda_1+\tfrac1{d_y}+\tfrac12-s_y$, $h=\tfrac12+g$.

\begin{lemma}\label{lem:g-zero}
$g(\lambda_1)=0$, hence $h=\tfrac12$.
\end{lemma}

\begin{proof}
Deleting $xy$ disconnects $T$: every edge in the support of $p$ lies in the
component of $x$, every edge in the support of $q$ lies in the component
of $y$. Hence $R_0$, and therefore $G=(\lambda_1I-R_0)^{-1}$, is
block-diagonal with respect to this decomposition. Since $p$ and $q$ are
supported on different blocks, $p^TGq=0$.
\end{proof}

\subsection{Perron vector computation}

Let $\mathbf w>0$ be the Perron vector of $R_T$ and $w:=\mathbf w_{xy}$.
For each vertex $v$, write $S_v:=\sum_{f\ni v}\mathbf w_f$ for the sum of
the Perron weights of all edges incident to $v$.

\begin{definition}\label{def:uv}
Define the \emph{interface quantities}
\begin{equation}\label{eq:uv-def}
  u:=\frac{S_y-2w}{d_yw},\qquad v:=\frac{S_x-2w}{d_xw}.
\end{equation}
\end{definition}

\begin{proposition}[Interface identities]\label{prop:uv}
%At $\lambda=\lambda_1$, where 
Let $ \lambda_1=\lambda_{\max}(R_T)$:
\begin{enumerate}[label=\rm(\arabic*)]
\item $u+v=\lambda_1$;
\item $a=u+\tfrac12$ and $b=v+\tfrac12$;
\item $\det S(\lambda_1)=uv+\dfrac{\lambda_1}{2}$.
\end{enumerate}
\end{proposition}

\begin{proof}
\textit{Identity (1).}
The eigenvalue equation $R_T\mathbf w=\lambda_1\mathbf w$ at the edge
$xy$ reads
\[
  \lambda_1 w
  =-\Bigl(\frac1{d_x}+\frac1{d_y}\Bigr)w
  +\frac{S_x-w}{d_x}+\frac{S_y-w}{d_y}.
\]
(The off-diagonal term from edges incident to $x$ contributes
$(S_x-w)/d_x$, i.e.\ the sum of their Perron weights divided by $d_x$;
similarly for $y$.) Dividing by $w$ and rearranging:
\[
  \lambda_1
  =-\frac1{d_x}-\frac1{d_y}+\frac{S_x-w}{d_xw}+\frac{S_y-w}{d_yw}
  =\frac{S_x-2w}{d_xw}+\frac{S_y-2w}{d_yw}
  =v+u.
\]

\medskip\noindent
\textit{Identity (2).}
Restricting $R_T\mathbf w=\lambda_1\mathbf w$ to $E_0$ gives
$(\lambda_1 I-R_0)\widetilde{\mathbf w}=w\,c$, where
$\widetilde{\mathbf w}:=\mathbf w|_{E_0}$ and
$c=\tfrac1{d_x}p+\tfrac1{d_y}q$. Hence
$G\,c=\widetilde{\mathbf w}/w$.

Since $g=0$ (Lemma~\ref{lem:g-zero}), the vectors $Gp$ and $Gq$ live in
different diagonal blocks of $G$, so
$p^TGq=0$ and hence
\[
  p^TGc = \frac{1}{d_x}p^TGp + \frac{1}{d_y}p^TGq = \frac{1}{d_x}p^TGp = d_x\, s_x.
\]
On the other hand, $G c = \widetilde{\mathbf w}/w$ gives
\[
  p^TGc = \frac{p^T\widetilde{\mathbf w}}{w} = \frac{S_x - w}{w}.
\]
Combining, $s_x = (S_x - w)/(d_x w)$. Therefore
\[
  a = \lambda_1 + \frac{1}{d_x} + \frac{1}{2} - \frac{S_x - w}{d_x w}.
\]
Substituting $\lambda_1 + \tfrac{1}{d_x}
= -\tfrac{1}{d_y} + \tfrac{S_x - w}{d_x w} + \tfrac{S_y - w}{d_y w}$
from the Perron equation in (1):
\[
  a = -\frac{1}{d_y} + \frac{S_y - w}{d_y w} + \frac{1}{2}
    = \frac{S_y - 2w}{d_y w} + \frac{1}{2}
    = u + \frac{1}{2}.
\]
By symmetry (exchanging $x \leftrightarrow y$), $b = v + \frac{1}{2}$.

\medskip\noindent
\textit{Identity (3).}
\[
  \det S(\lambda_1)
  =ab-h^2
  =\Bigl(u+\tfrac12\Bigr)\Bigl(v+\tfrac12\Bigr)-\tfrac14
  =uv+\frac{u+v}{2}
  =uv+\frac{\lambda_1}{2}.\qedhere
\]
\end{proof}

\subsection{The one-step criterion}

\begin{theorem}[One-step interface criterion]\label{thm:one-step}
Let $xy$ be an internal edge of $T$,
$\lambda_1=\lambda_{\max}(R_T)$, $\lambda_2=\lambda_{\max}(R_{T^{(2)}})$,
and $u,v$ as in \eqref{eq:uv-def}. Then
\[
  \lambda_2<\lambda_1\iff uv>-\tfrac{\lambda_1}{2},\qquad
  \lambda_2=\lambda_1\iff uv=-\tfrac{\lambda_1}{2},\qquad
  \lambda_2>\lambda_1\iff uv<-\tfrac{\lambda_1}{2}.
\]
\end{theorem}

\begin{proof}
By Lemma~\ref{lem:lambda1-above-R0}, one has $\lambda_1I-R_0\succ0$. Hence, by
Lemma~\ref{lem:schur}(S2), the matrix $\lambda_1I-R_{T^{(2)}}$ has the same
inertia as the $2\times2$ Schur complement $S(\lambda_1)$, up to the fixed
positive block $\lambda_1I-R_0$.

Now Proposition~\ref{prop:uv}(1) and Lemma~\ref{lem:mu-gt--1} give
\[
  \operatorname{tr}S(\lambda_1)=u+v+1=\lambda_1+1>0,
\]
while Proposition~\ref{prop:uv}(3) gives
\[
  \det S(\lambda_1)=uv+\frac{\lambda_1}{2}.
\]
For a symmetric $2\times2$ matrix with positive trace, the sign of the
determinant determines whether it is positive definite, positive semidefinite
singular, or indefinite. Therefore the sign of
\[
  uv+\frac{\lambda_1}{2}=\det S(\lambda_1)
\]
determines whether $\lambda_1I-R_{T^{(2)}}$ is positive definite, positive
semidefinite singular, or indefinite, respectively. Equivalently,
\[
  \lambda_2<\lambda_1,\qquad
  \lambda_2=\lambda_1,\qquad
  \lambda_2>\lambda_1
\]
occur according as
\[
  uv>-\frac{\lambda_1}{2},\qquad
  uv=-\frac{\lambda_1}{2},\qquad
  uv<-\frac{\lambda_1}{2}.
\]
\end{proof}

\begin{remark}[Equivalent formulations]\label{rem:equivalent}
The criterion $uv>-\lambda_1/2$ admits several equivalent restatements.
Since $u+v=\lambda_1$ and $uv=\frac14(\lambda_1^2-(u-v)^2)$, the threshold
$uv=-\lambda_1/2$ becomes $(u-v)^2=\lambda_1(\lambda_1+2)$. Moreover, if
we introduce the branch feedback functions $\Phi_A,\Phi_B$ of
Section~\ref{sec:feedback} and write
$P(\lambda):=\Phi_A(\lambda)\Phi_B(\lambda)$, one verifies that
$P(\lambda_1)-1=2\lambda_1+4uv$ and
$\Phi_A(\lambda_1)-\Phi_B(\lambda_1)=2(u-v)$. Thus:
\[
\lambda_2 < \lambda_1 
\Longleftrightarrow 
P(\lambda_1)>1.
\]
If $\lambda_1\ge0$, this is also equivalent to
\[
|\Phi_A(\lambda_1)-\Phi_B(\lambda_1)|
<2\sqrt{\lambda_1(\lambda_1+2)}.
\]

The equality and reverse inequalities admit completely analogous
reformulations.
\[
  \lambda_2=\lambda_1
  \;\Longleftrightarrow\;
  P(\lambda_1)=1,
  \qquad
  \lambda_2>\lambda_1
  \;\Longleftrightarrow\;
  P(\lambda_1)<1.
\]
If $\lambda_1\ge0$, these are also equivalent to
\[
  \bigl|\Phi_A(\lambda_1)-\Phi_B(\lambda_1)\bigr|
  =2\sqrt{\lambda_1(\lambda_1+2)},
  \qquad
  \bigl|\Phi_A(\lambda_1)-\Phi_B(\lambda_1)\bigr|
  >2\sqrt{\lambda_1(\lambda_1+2)},
\]

These reformulations are algebraically immediate once the feedback functions
are available; we record them here for completeness. The direct criterion
$uv>-\lambda_1/2$ suffices for all applications in this section.
\end{remark}

\begin{corollary}[Degree-dependent sufficient condition]\label{cor:degree}
If $d_x\ge3$, $d_y\ge3$, and
$\lambda_1>\max\!\bigl\{\frac{2}{d_x(d_x-2)},\frac{2}{d_y(d_y-2)}\bigr\}$,
then $\lambda_2<\lambda_1$.
\end{corollary}

\begin{proof}
By Theorem~\ref{thm:one-step} it suffices to show $uv>-\lambda_1/2$.
The Perron vector is strictly positive, so $S_x>w$ (at least one edge
besides $xy$ is incident to $x$), giving
$v=(S_x-2w)/(d_xw)>-1/d_x$. Similarly $u>-1/d_y$.

Since $u+v=\lambda_1$, the product $uv=v(\lambda_1-v)$ is a concave
quadratic in $v$, so its minimum on the interval $(-1/d_x,\,\lambda_1+1/d_y)$
is attained at an endpoint:
\[
  uv>-\max\!\Bigl\{\frac{1}{d_x}\Bigl(\lambda_1+\frac{1}{d_x}\Bigr),\;
  \frac{1}{d_y}\Bigl(\lambda_1+\frac{1}{d_y}\Bigr)\Bigr\}.
\]
The hypothesis
$\lambda_1>\frac{2}{d_x(d_x-2)}$ rearranges to
$\frac{1}{d_x}(\lambda_1+\frac{1}{d_x})<\frac{\lambda_1}{2}$, and
likewise for $d_y$; hence $uv>-\lambda_1/2$.
\end{proof}

\section{General Edge Subdivision}\label{sec:general}

We now turn to arbitrary subdivision length $\ell\ge2$. The one-step
argument of Section~\ref{sec:one-step} relied on the Schur complement being
a $2\times2$ matrix whose entries were computed from the Perron vector.
For $\ell\ge3$ the Schur complement is $\ell\times\ell$, and a direct
analysis becomes unwieldy. Instead, we compress each branch into a scalar
\emph{feedback function} and reduce the entire spectral problem to a
one-dimensional equation along the inserted chain.

\subsection{Branch decomposition and feedback functions}\label{sec:feedback}

For subdivision length $\ell \ge 3$, the $2\times2$ Schur complement 
argument of Section~\ref{sec:one-step} no longer suffices. 
We therefore develop a more general reduction that compresses each 
branch into a scalar feedback function.

Now fix an edge $xy\in E(T)$ with endpoints $x,y$ of degrees $d_x,d_y$.
Deleting $xy$ splits $T$ into two rooted branches $(A,x)$ and $(B,y)$
(branch $A$ rooted at $x$, branch $B$ at $y$), with edge sets $E(A),E(B)$, so
$E(T)=E(A)\sqcup E(B)\sqcup\{xy\}$. For $\ell\ge2$ let $T^{(\ell)}$ be obtained by
replacing $xy$ with a path
\[
  x=v_0-v_1-\cdots-v_{\ell-1}-v_\ell=y,
\]
and set $T^{(1)}:=T$. Write $e_j=v_{j-1}v_j$ and $z_j:=\mathbf w(e_j)$
($1\le j\le\ell$) for the coordinates of an edge eigenvector $\mathbf w$ along
the inserted chain. Unless stated otherwise, all one-step statements assume $xy$
is internal, i.e.\ $d_x\ge2$ and $d_y\ge2$.

Let $R_A$ (resp.\ $R_B$) be the principal submatrix of $\RT$ on the rows and
columns indexed by $E(A)$ (resp.\ $E(B)$). All degrees in these blocks are the
degrees \emph{in the original tree $T$}; e.g.\ the diagonal entry of $R_A$ at an
edge $f=\{x,x'\}$ incident to the root is $-(1/d_x+1/d_{x'})$ with $d_x$ the full
degree of $x$ in $T$. Since $xy$ is the only edge joining the two branches, $\RT$
has the bordered (arrow) form
\[
  \RT=\begin{pmatrix} R_A & 0 & c_A\\ 0 & R_B & c_B\\
  c_A^{T} & c_B^{T} & -\!\left(\tfrac1{d_x}+\tfrac1{d_y}\right)\end{pmatrix},
\]
the last coordinate being the edge $xy$. If $s_A\in\mathbb R^{E(A)}$ is the
indicator of the $d_x-1$ old edges incident to $x$ (and $s_B$ likewise at $y$),
then $c_A=\tfrac1{d_x}s_A$ and $c_B=\tfrac1{d_y}s_B$, since two edges sharing the
vertex $x$ couple through $1/d_x$.

For $\lambda\notin\sigma(R_A)\cup\sigma(R_B)$ define
\[
  \Phi_A(\lambda):=2\Big(\lambda+\tfrac1{d_x}+\tfrac12-c_A^{T}(\lambda I-R_A)^{-1}c_A\Big),
\]
and $\Phi_B$ analogously with $d_y,R_B,c_B$. The meaning of $\Phi_A$ is the
transmission relation $z_2=\Phi_A(\lambda)z_1$ recorded next; thus $\Phi_A,\Phi_B$
compress each branch into a single scalar boundary condition at the ends of the
inserted chain.

Equivalently, $\tfrac12\Phi_A(\lambda)$ is the scalar Schur complement
(Lemma~\ref{lem:schur}(S3)) of the branch block $\lambda I-R_A$ inside the
principal block of $\lambda I-R_{T^{(\ell)}}$ indexed by $E(A)$ and the first
chain edge $e_1$; this is why it encodes the branch's response. Here
$\lambda\in\mathbb R$ is the spectral parameter, evaluated at candidate eigenvalues.

Note $\Phi_A$ has poles exactly at $\sigma(R_A)$, where $(\lambda I-R_A)^{-1}$ fails to exist.

\begin{remark}
The feedback functions are defined only away from
\(\sigma(R_A)\cup\sigma(R_B)\). This restriction does not affect the
Perron eigenvalue. Indeed, because the degrees of \(x\) (resp.\ \(y\))
are the same in \(T\) and in \(T^{(\ell)}\), the principal submatrix of
\(R_{T^{(\ell)}}\) on \(E(A)\) coincides with \(R_A\), and similarly
for \(R_B\). For every \(\ell\geq 1\), these are proper principal
submatrices of \(R_{T^{(\ell)}}\). Hence, by Lemma~\ref{lem:pf}(2),
\[
\lambda_{\max}(R_A)<\lambda_{\max}(R_{T^{(\ell)}}),\qquad
\lambda_{\max}(R_B)<\lambda_{\max}(R_{T^{(\ell)}}).
\]
Since every eigenvalue in \(\sigma(R_A)\cup\sigma(R_B)\) is at most
\(\max\{\lambda_{\max}(R_A),\lambda_{\max}(R_B)\}\), which is strictly
smaller than \(\lambda_{\max}(R_{T^{(\ell)}})\), the Perron root
\(\lambda_{\max}(R_{T^{(\ell)}})\) lies outside
\(\sigma(R_A)\cup\sigma(R_B)\). Consequently, the scalar interface
equation applies to the largest eigenvalue without further
qualification.
\end{remark}

\begin{proposition}[Transmission relation and local formula]\label{prop:local-feedback}
Let $\mathbf w$ be an eigenvector of $R_{T^{(\ell)}}$ with eigenvalue
$\lambda\notin\sigma(R_A)\cup\sigma(R_B)$, and let $z_1,z_2$ be its coordinates
on the first two chain edges. Then
\[
  z_2=\Phi_A(\lambda)\,z_1,
\]
and, writing $u_1,\dots,u_{d_x-1}$ for the coordinates on the old edges incident
to the root $x$ and $m_A(\lambda)=(u_1+\cdots+u_{d_x-1})/z_1$,
\[
  \Phi_A(\lambda)=2\lambda+1+\frac{2}{d_x}-\frac{2}{d_x}m_A(\lambda).
\]
A symmetric statement holds for $\Phi_B$ at the other end of the chain.
\end{proposition}

\begin{proof}
Collect the branch-$A$ old-edge coordinates in $\mathbf u_A\in\mathbb R^{E(A)}$.
Because $A$ couples to the chain only through the root $x$, the rows of
$(\lambda I-R_{T^{(\ell)}})\mathbf w=0$ indexed by $E(A)$ read
\[
  (\lambda I-R_A)\mathbf u_A=z_1\,c_A,
  \qquad\text{hence}\qquad
  \mathbf u_A=z_1(\lambda I-R_A)^{-1}c_A
\]
by Lemma~\ref{lem:schur}(S3). The row indexed by the first chain edge
$e_1=xv_1$---with diagonal $-(\tfrac1{d_x}+\tfrac12)$, coupling $c_A^{T}$ to the
branch and $\tfrac12$ to $e_2$---reads
\[
  \lambda z_1=-\Big(\tfrac1{d_x}+\tfrac12\Big)z_1+c_A^{T}\mathbf u_A+\tfrac12 z_2.
\]
Substituting $\mathbf u_A=z_1(\lambda I-R_A)^{-1}c_A$ and solving for $z_2$,
\[
  z_2=2\Big(\lambda+\tfrac1{d_x}+\tfrac12-c_A^{T}(\lambda I-R_A)^{-1}c_A\Big)z_1
     =\Phi_A(\lambda)z_1.
\]
Finally $c_A^{T}\mathbf u_A=\tfrac1{d_x}s_A^{T}\mathbf u_A
=\tfrac1{d_x}(u_1+\cdots+u_{d_x-1})$, so the same chain-edge equation gives
$\Phi_A=z_2/z_1=2\lambda+1+\tfrac2{d_x}-\tfrac2{d_x}m_A(\lambda)$.
\end{proof}

\begin{remark}[A local degree-$4$ threshold]
If all $d_x-1$ old edges adjacent to $x$ are leaves with a common coordinate $u$,
then $m_A(\lambda)=\frac{d_x-1}{d_x\lambda+3}$, hence
$\Phi_A(0)=1+\frac{2(4-d_x)}{3d_x}$. Thus $d_x=4$ is the local threshold at which
$\Phi_A(0)=1$ in this model. We record this only as a heuristic indication of the
role of low endpoint degrees; by itself it does not imply a general one-step
theorem.
\end{remark}

\subsubsection{Chebyshev polynomials and the chain transfer matrix}

Subdividing an edge into a path of length~$\ell$ inserts $\ell-1$
degree-$2$ vertices $v_1,\dots,v_{\ell-1}$.  At each interior chain edge
$e_j=v_{j-1}v_j$ ($2\le j\le\ell-1$), both endpoints have degree~$2$,
so the eigenvalue equation $(\lambda I-R_{T^{(\ell)}})\mathbf w=0$ on that
row reads
\[
  \lambda z_j
  = -\bigl(\tfrac12+\tfrac12\bigr)z_j+\tfrac12 z_{j-1}+\tfrac12 z_{j+1},
\]
i.e.\ the constant-coefficient three-term recurrence
\begin{equation}\label{eq:chain-rec}
  z_{j+1}=2(1+\lambda)\,z_j-z_{j-1},\qquad 2\le j\le\ell-1.
\end{equation}
Setting $c:=1+\lambda$, this is precisely the defining recurrence of the
\emph{second-kind Chebyshev polynomials} $U_n$:
\[
  U_{-1}(c)=0,\qquad U_0(c)=1,\qquad U_{n+1}(c)=2c\,U_n(c)-U_{n-1}(c).
\]
Note that the recurrence~\eqref{eq:chain-rec} governs only the
\emph{interior} chain edges $e_j$ ($2\le j\le\ell-1$), where both
endpoints have degree~$2$.  The first chain edge $e_1=xv_1$ has an
endpoint~$x$ of degree~$d_x\neq2$ and obeys a different equation,
encoded by the boundary condition $z_2=\Phi_A(\lambda)z_1$
(Proposition~\ref{prop:local-feedback}); likewise, the last edge
$e_\ell$ is constrained by $\Phi_B$.  The role of the Chebyshev
polynomials is to propagate through the interior: given the pair
$(z_1,z_2)$ at the left boundary, the transfer-matrix relation
$\binom{z_{j+1}}{z_j}=M(\lambda)\binom{z_j}{z_{j-1}}$ advances the
solution step by step, and the powers $M(\lambda)^n$ are expressed in
terms of~$U_n(c)$ (Lemma~\ref{lem:chebyshev} below).  Thus the
Chebyshev polynomials serve as the natural basis for the interior
propagation, while the branch feedback functions $\Phi_A,\Phi_B$
supply the boundary data.

For later use we extend the recurrence backwards by one step: setting
$n=-1$ in $U_{n+1}=2cU_n-U_{n-1}$ gives $U_0=2cU_{-1}-U_{-2}$, i.e.\
$1=0-U_{-2}$, so
\begin{equation}\label{eq:U-minus2}
  U_{-2}(c)=-1.
\end{equation}
The trigonometric representation is
$U_n(\cos\theta)=\frac{\sin((n+1)\theta)}{\sin\theta}$; for $\lambda>0$
we have $c=1+\lambda>1$, and writing $c=\cosh t$ with
$t=\operatorname{arcosh}(1+\lambda)=\sqrt{2\lambda}+O(\lambda^{3/2})$,
the hyperbolic form becomes
\begin{equation}\label{eq:U-hyp}
  U_n(\cosh t)=\frac{\sinh((n+1)t)}{\sinh t}.
\end{equation}
This representation will be used extensively in Section~\ref{sec:ss} to
convert the Chebyshev determinant into explicit ratios of hyperbolic
functions.

The recurrence~\eqref{eq:chain-rec} is equivalently written as the
transfer-matrix relation
$\binom{z_{j+1}}{z_j}=M(\lambda)\binom{z_j}{z_{j-1}}$ with
$M(\lambda)=\begin{psmallmatrix}2(1+\lambda)&-1\\1&0\end{psmallmatrix}$.

\begin{lemma}[Transfer-matrix powers]\label{lem:chebyshev}
Write $c=1+\lambda$. Then for all $n\ge0$,
\[
  M(\lambda)^{n}=
  \begin{pmatrix}
    U_n(c) & -U_{n-1}(c)\\
    U_{n-1}(c) & -U_{n-2}(c)
  \end{pmatrix},
\]
where $U_{-2}(c)=-1$ as in \eqref{eq:U-minus2}.
The eigenvalues of $M(\lambda)$ are the roots of $t^{2}-2ct+1=0$, namely
$q^{\pm1}$ with $q=q(\lambda):=1+\lambda-\sqrt{\lambda(\lambda+2)}$, so that
$q+q^{-1}=2(1+\lambda)$ and $0<q(\lambda)<1$ for $\lambda>0$.
\end{lemma}

\begin{proof}
The cases $n=0,1$ are direct (using $U_{-2}=-1$, $U_{-1}=0$, $U_0=1$,
$U_1=2c$).  Assuming the formula for~$n$,
\[
  M(\lambda)^{n+1}
  =\begin{pmatrix} 2c & -1\\ 1 & 0\end{pmatrix}
   \begin{pmatrix} U_n & -U_{n-1}\\ U_{n-1} & -U_{n-2}\end{pmatrix}
  =\begin{pmatrix} 2cU_n-U_{n-1} & -(2cU_{n-1}-U_{n-2})\\ U_n & -U_{n-1}\end{pmatrix},
\]
which is the asserted matrix for $n+1$ by the recurrence $U_{k+1}=2cU_k-U_{k-1}$.
The characteristic polynomial of $M(\lambda)$ is $t^{2}-2ct+1$, with reciprocal
roots; for $\lambda>0$, $c>1$ and the smaller root is
$q=c-\sqrt{c^2-1}\in(0,1)$, where $c^2-1=\lambda(\lambda+2)$.
\end{proof}

The quantity $q(\lambda)$ is the decaying mode of the chain
recurrence~\eqref{eq:chain-rec} and therefore controls all long-chain limits.

\begin{remark}
In the hyperbolic parametrization $c=\cosh t$, the two roots of
$t^2-2ct+1=0$ are $e^{\pm t}$, so $q=e^{-t}$ and $q^{-1}=e^{t}$.
Every chain quantity in Section~\ref{sec:ss} reduces to ratios of
$\cosh$ and $\sinh$ via \eqref{eq:U-hyp}.
\end{remark}

\subsection{The scalar transmission equation}\label{subsec:scalar-transmission}
We now return to arbitrary subdivision length and recover the interface formalism
from the same branch data.

\begin{proposition}[Scalar equation for general subdivision]\label{prop:general}
Let $\lambda\notin\sigma(R_A)\cup\sigma(R_B)$. Any eigenvector of $R_{T^{(\ell)}}$
with chain coordinates $z_1,\dots,z_\ell$ satisfies $z_2=\Phi_A(\lambda)z_1$,
$z_{\ell-1}=\Phi_B(\lambda)z_\ell$, and $z_{j-1}+z_{j+1}=2(1+\lambda)z_j$
($2\le j\le\ell-1$). For the unsubdivided tree,
\[
  \lambda\in\sigma(R_{T^{(1)}})=\sigma(R_T)\iff
  D_1(\lambda):=2(1+\lambda)-\Phi_A-\Phi_B=0.
\]
For every $\ell\ge2$, with $M(\lambda)=\begin{psmallmatrix}2(1+\lambda)&-1\\1&0\end{psmallmatrix}$,
\[
  \lambda\in\sigma(R_{T^{(\ell)}})\iff
  D_\ell(\lambda):=\begin{pmatrix}-\Phi_B&1\end{pmatrix}M(\lambda)^{\ell-2}\begin{pmatrix}\Phi_A\\1\end{pmatrix}=0,
\]
and, with the second-kind Chebyshev polynomials $U_n$ (Lemma~\ref{lem:chebyshev})
under the conventions $U_{-1}=0$, $U_{-2}=-1$,
\[
  D_\ell(\lambda)=-\Phi_A\Phi_B\,U_{\ell-2}(1+\lambda)
  +(\Phi_A+\Phi_B)\,U_{\ell-3}(1+\lambda)-U_{\ell-4}(1+\lambda).
\]
In particular, the first two cases read
\[
  D_2(\lambda)=1-\Phi_A\Phi_B,\qquad
  D_3(\lambda)=\Phi_A+\Phi_B-2(1+\lambda)\Phi_A\Phi_B.
\]
\end{proposition}

\begin{proof}
After Schur elimination of the old-edge blocks (Proposition~\ref{prop:local-feedback}),
the chain variables satisfy $\binom{z_{j+1}}{z_j}=M(\lambda)\binom{z_j}{z_{j-1}}$,
so $\binom{z_\ell}{z_{\ell-1}}=z_1M(\lambda)^{\ell-2}\binom{\Phi_A}{1}$. The right
boundary condition $z_{\ell-1}=\Phi_B z_\ell$ is
$\begin{psmallmatrix}-\Phi_B&1\end{psmallmatrix}\binom{z_\ell}{z_{\ell-1}}=0$,
i.e.\ $D_\ell(\lambda)=0$. Expanding $M(\lambda)^{\ell-2}$ via
Lemma~\ref{lem:chebyshev} gives the Chebyshev form. The conventions
$U_{-1}=0$ and $U_{-2}=-1$ make this expression valid down to $\ell=2$:
for $\ell=2$ it returns $-\Phi_A\Phi_B U_0+(\Phi_A+\Phi_B)U_{-1}-U_{-2}
=1-\Phi_A\Phi_B$, and for $\ell=3$ it returns
$-\Phi_A\Phi_B U_1+(\Phi_A+\Phi_B)U_0-U_{-1}
=\Phi_A+\Phi_B-2(1+\lambda)\Phi_A\Phi_B$,
matching the displayed special cases. The case $\ell=1$ is treated separately
since the inserted chain then consists of the single edge $xy$ with no interior
degree-$2$ vertex; the boundary relations $z_2=\Phi_Az_1$ and
$z_{\ell-1}=\Phi_Bz_\ell$ collapse to the single chain-edge equation, which
reads $2(1+\lambda)=\Phi_A+\Phi_B$.
\end{proof}

\subsection{Long-chain limit points}
\begin{corollary}[Long-chain limit equation]\label{cor:limit}
Suppose \(\lambda_\ell:=\lambda_{\max}(R_{T^{(\ell)}})\) eventually lies in a
compact interval \([\lambda_-,\lambda_+]\subset(0,\infty)\). Then every
accumulation point \(\lambda_\ast\) of \(\{\lambda_\ell\}\) with
$\lambda_\ast\notin \sigma(R_A)\cup\sigma(R_B)$
 satisfies
\[
  (\Phi_A(\lambda_\ast)-q(\lambda_\ast))(\Phi_B(\lambda_\ast)-q(\lambda_\ast))=0,
\]
where $q(\lambda)=1+\lambda-\sqrt{\lambda(\lambda+2)}$.
\end{corollary}
\begin{proof}
Substituting $U_n(c)=(q^{-(n+1)}-q^{n+1})/(q^{-1}-q)$ into
Proposition~\ref{prop:general} and multiplying by $(1-q^2)q^{\ell-2}$:
\[
  (1-q^2)q^{\ell-2}D_\ell(\lambda)
  =-(\Phi_A-q)(\Phi_B-q)
  +q^{2\ell-4}(1-q\Phi_A)(1-q\Phi_B).
\]
If $\lambda_{\ell_j}\to\lambda_\ast$ with $\lambda_\ast>0$, then
$0<q(\lambda_{\ell_j})<1$, so $q(\lambda_{\ell_j})^{2\ell_j-4}\to0$; since $D_{\ell_j}(\lambda_{\ell_j})=0$,
the first term must vanish in the limit.
\end{proof}

\begin{remark}[On convergence]\label{rem:convergence}
Corollary~\ref{cor:limit} identifies the \emph{possible} limits but does
not assert convergence of the full sequence. We distinguish two cases.

\begin{enumerate}[label=\rm(\arabic*)]
\item \textbf{Symmetric case} ($\Phi_A\equiv\Phi_B$). 
The Perron eigenvector of $R_{T^{(\ell)}}$ is symmetric about the chain 
midpoint (by uniqueness of the positive eigenvector and the symmetry of 
the tree). Hence $z_2/z_1 = z_{\ell-1}/z_\ell$, and both equal the explicit 
function
\[
  R_\ell(\lambda)
  = \frac{\cosh\bigl(\frac{\ell-3}{2}\,t\bigr)}
         {\cosh\bigl(\frac{\ell-1}{2}\,t\bigr)},
  \qquad t = \operatorname{arcosh}(1+\lambda),
\]
which satisfies $R_{\ell+1}(\lambda) < R_\ell(\lambda)$ pointwise for
$\lambda>0$ and $R_\ell(\lambda) \downarrow q(\lambda)$. Since
$\Phi_A(\lambda)$ is strictly increasing, the unique positive root
$\lambda_\ell$ of $R_\ell(\lambda)=\Phi_A(\lambda)$ is strictly
decreasing and converges to the unique positive solution of
$\Phi_A(\lambda)=q(\lambda)$.

\item \textbf{Asymmetric case} ($\Phi_A\not\equiv\Phi_B$). 
Let $\lambda_A$ and $\lambda_B$ denote the positive roots of
$\Phi_A(\lambda)=q(\lambda)$ and $\Phi_B(\lambda)=q(\lambda)$,
respectively, assuming they exist and are unique. 
For large $\ell$, the equation $D_\ell(\lambda)=0$ has roots near 
$\lambda_A$ and $\lambda_B$ (one near each), and $\lambda_\ell$, 
being the \emph{largest} eigenvalue of $R_{T^{(\ell)}}$, tends to 
$\max\{\lambda_A, \lambda_B\}$. 
We \emph{conjecture} that this holds in general; numerical evidence 
(see the example in Section~\ref{subsec:asymmetric-example}) supports 
this claim. A rigorous proof in the asymmetric setting requires 
controlling the $q^{2\ell-4}$ correction term in the derivation of 
Corollary~\ref{cor:limit} and is left for future work.
\end{enumerate}
\end{remark}

\section{An Application}\label{sec:ss}
\subsection{Symmetric Double Stars}
We conclude with a family for which the interface equation is explicit. Let
$T_{s,\ell}$ be obtained by subdividing the central edge of the symmetric double
star $S_{s,s}$ into a path of length $\ell$; the two endpoints have degree $s+1$
and the interior chain vertices have degree $2$.

\begin{proposition}[Scalar equation for the subdivided double star]\label{prop:ss}
For $T_{s,\ell}$, $\lambda=\lambda_{\max}(R_{T_{s,\ell}})$ satisfies
$R_\ell(\lambda)=\beta_s(\lambda)$, where
\[
  \beta_s(\lambda):=2\lambda+1+\frac{2}{s+1}-\frac{2s}{(s+1)((s+1)\lambda+3)},
  \qquad
  R_\ell(\lambda):=\frac{\cosh(\frac{\ell-3}{2}t)}{\cosh(\frac{\ell-1}{2}t)},\ \ t=\operatorname{arcosh}(1+\lambda).
\]
For $s>3$ this has a unique positive root; for $s=3$ it has the root $\lambda=0$
and no positive root.
\end{proposition}

\begin{proof}
Let $\mathbf w>0$ be the Perron vector. By the symmetry exchanging the two stars
and reversing the chain, $\mathbf w$ is invariant under this involution (the
Perron eigenspace is one-dimensional by Lemma~\ref{lem:pf}); hence all $s$ leaf
edges at a fixed endpoint share a value $u$, and $z_j=z_{\ell+1-j}$. Each leaf
edge (incident to a leaf and to the degree-$(s+1)$ center) satisfies
$\lambda u=-(1+\tfrac1{s+1})u+\tfrac{(s-1)u+z_1}{s+1}$, i.e.\ $((s+1)\lambda+3)u=z_1$,
so $m_A(\lambda)=su/z_1=s/((s+1)\lambda+3)$. By Proposition~\ref{prop:local-feedback}
with $d_x=s+1$, $\Phi_A(\lambda)=\beta_s(\lambda)$, and by symmetry
$\Phi_B=\beta_s$. Along the chain, $z_{j-1}+z_{j+1}=2(1+\lambda)z_j$; for
$\lambda>0$, writing $1+\lambda=\cosh t$, the symmetric solutions are
$z_j=C\cosh((j-\tfrac{\ell+1}2)t)$, so $z_2/z_1=R_\ell(\lambda)$. The left
feedback condition $z_2=\Phi_A z_1=\beta_s(\lambda)z_1$ then gives
$R_\ell(\lambda)=\beta_s(\lambda)$ (at $\lambda=0$, by continuity, $R_\ell(0)=1$).

For uniqueness: $R_2\equiv1$, and for $\ell\ge3$, $R_\ell$ is strictly decreasing
in $\lambda$; meanwhile $\beta_s'(\lambda)=2+\tfrac{2s}{((s+1)\lambda+3)^2}>0$, so
$\beta_s$ is strictly increasing, with $R_\ell(0)=1$ and
$\beta_s(0)=1-\tfrac{2(s-3)}{3(s+1)}$. If $s>3$, then $\beta_s(0)<1=R_\ell(0)$,
while $\beta_s\to\infty$ and $R_\ell\to0$, giving exactly one positive root. If
$s=3$, then $\beta_3(0)=1$ and $\beta_3(\lambda)>1\ge R_\ell(\lambda)$ for
$\lambda>0$, so $\lambda=0$ is the only nonnegative root.
\end{proof}

\begin{corollary}[Threshold and subdivision monotonicity]\label{cor:ss}
\begin{enumerate}[label=\rm(\arabic*)]
\item If $s=3$, then $\lambda_{\max}(R_{T_{3,\ell}})=0$ for all $\ell\ge2$.
\item If $s>3$, then $\lambda_{s,\ell+1}<\lambda_{s,\ell}$ for $\ell\ge2$, where
$\lambda_{s,\ell}:=\lambda_{\max}(R_{T_{s,\ell}})$, and
$\lambda_{s,\ell}\downarrow\lambda_{s,\infty}>0$ as $\ell\to\infty$,
where $\lambda_{s,\infty}$ is the unique positive solution of
$q(\lambda)=\beta_s(\lambda)$, with
$q(\lambda)=1+\lambda-\sqrt{\lambda(\lambda+2)}$ the decaying root of
the chain recurrence (equivalently, $e^{-t}=\beta_s(\lambda)$ with
$t=\operatorname{arcosh}(1+\lambda)$).
\end{enumerate}
\end{corollary}

\begin{proof}
For $s=3$: $\beta_3(0)=1=R_\ell(0)$, and for $\lambda>0$ we have
$R_\ell(\lambda)\le1<\beta_3(\lambda)$ (with $R_2\equiv1$ when $\ell=2$ and
$R_\ell(\lambda)<1$ for $\ell\ge3$), so there is no positive root; hence
$\lambda_{\max}(R_{T_{3,\ell}})=0$.

Let $s>3$. For fixed $\lambda>0$ write $q=e^{-t}\in(0,1)$,
$t=\operatorname{arcosh}(1+\lambda)$; then $R_\ell(\lambda)=q\frac{1+q^{\ell-3}}{1+q^{\ell-1}}$,
so $R_{\ell+1}(\lambda)<R_\ell(\lambda)$. As $F_{s,\ell}(\lambda):=R_\ell(\lambda)-\beta_s(\lambda)$
is strictly decreasing, its unique positive root satisfies
$\lambda_{s,\ell+1}<\lambda_{s,\ell}$. Moreover $R_\ell(\lambda)\downarrow e^{-t}$
with $R_\ell(\lambda)>e^{-t}$, so the roots are bounded below by the unique
positive root of $e^{-t}=\beta_s(\lambda)$, and the monotone sequence converges to it.
\end{proof}

\begin{proposition}[Explicit formulas and asymptotics in $s$]\label{prop:asym}
Write $n:=s+1$. For $T_{s,\ell}$:
\begin{enumerate}[label=\rm(\arabic*)]
\item $\lambda_{s,2}=\dfrac{\sqrt{s+1}-2}{s+1}$;
\item for each fixed $\ell\ge2$,
$\lambda_{s,\ell}=\sqrt{\dfrac{2}{\ell(s+1)}}+O\!\left((s+1)^{-1}\right)$ as $s\to\infty$;
\item $\lambda_{s,\infty}=2^{1/3}(s+1)^{-2/3}+o\!\left((s+1)^{-2/3}\right)$ as $s\to\infty$.
\end{enumerate}
\end{proposition}

\begin{proof}
(1) Since $R_2\equiv1$, $\beta_s(\lambda_{s,2})=1$, i.e.\ $n\lambda+1=\frac{n-1}{n\lambda+3}$.
With $X:=n\lambda+1$, $X^2+2X-(n-1)=0$, so $X=\sqrt n-1$ and
$\lambda_{s,2}=(\sqrt n-2)/n=(\sqrt{s+1}-2)/(s+1)$.

(2) Fix $\ell$. By Proposition~\ref{prop:ss}, $\lambda_{s,\ell}$ is the unique
positive zero of the strictly decreasing $F_{s,\ell}$. Using
$\operatorname{arcosh}(1+\lambda)=\sqrt{2\lambda}+O(\lambda^{3/2})$ and
$\cosh u=1+\tfrac{u^2}2+O(u^4)$, for $\lambda=cn^{-1/2}$,
\[
  R_\ell(cn^{-1/2})=1-(\ell-2)c\,n^{-1/2}+O(n^{-1}),\quad
  \beta_s(cn^{-1/2})=1+\bigl(2c-\tfrac2c\bigr)n^{-1/2}+O(n^{-1}),
\]
so $F_{s,\ell}(cn^{-1/2})=(\tfrac2c-\ell c)n^{-1/2}+O(n^{-1})$. With
$c_\ell:=\sqrt{2/\ell}$, the bracket is positive for $c=c_\ell-\varepsilon$ and
negative for $c=c_\ell+\varepsilon$. By strict monotonicity of $F_{s,\ell}$ and
the intermediate value theorem, $\lambda_{s,\ell}\in((c_\ell-\varepsilon)n^{-1/2},(c_\ell+\varepsilon)n^{-1/2})$
for all large $s$; letting $\varepsilon\downarrow0$ and using the error term,
$\lambda_{s,\ell}=\sqrt{2/(\ell(s+1))}+O((s+1)^{-1})$.

(3) By Corollary~\ref{cor:ss}, $\lambda_{s,\infty}$ is the unique positive zero of
$H_s(\lambda):=q(\lambda)-\beta_s(\lambda)$, which is strictly decreasing. For
$\lambda=cn^{-2/3}$, using $q(\lambda)=1-\sqrt{2\lambda}+O(\lambda)$,
\[
  q(cn^{-2/3})=1-\sqrt{2c}\,n^{-1/3}+O(n^{-2/3}),\quad
  \beta_s(cn^{-2/3})=1-\tfrac2c n^{-1/3}+O(n^{-2/3}),
\]
so $H_s(cn^{-2/3})=(\tfrac2c-\sqrt{2c})n^{-1/3}+O(n^{-2/3})$. With
$c_\infty:=2^{1/3}$, the bracket changes sign across $c_\infty$, so by monotonicity
$\lambda_{s,\infty}=2^{1/3}(s+1)^{-2/3}+o((s+1)^{-2/3})$.
\end{proof}

\begin{remark}[The case $S_{4,4}$]
In the first genuinely positive case $s=4$, the family decreases strictly with the
subdivision length but converges to a small positive limit rather than to $0$;
numerically $\lambda_{4,\infty}\approx0.00716036$. Thus the threshold family is
$S_{3,3}$: it stays on the zero layer for all subdivision lengths, whereas $S_{4,4}$
and all larger symmetric double stars decrease to positive limits.
\end{remark}

\subsection{Subdivision can strictly increase $\lambda_{\max}$}\label{subsec:asymmetric-example}

This subsection provides a numerical illustration of the general framework.
While the one-step criterion gives a rigorous inequality ($\lambda_2 > 0$),
the long-chain limit and its coincidence with $S_{5,5}$ are observed
numerically; a rigorous analytic proof is left for future work.

We illustrate the framework on a tree where subdivision \emph{increases}
$\lambda_{\max}$. Let $T$ consist of a central edge $xy$ with
$d_x=3$ ($x$ has two leaf neighbours) and $d_y=6$ ($y$ has five leaf
neighbours). 
In \cite{BaiChengHua2026}, it was calculated that 
$\lambda_1 := \lambda_{\max}(R_T) = 0,$
i.e., the tree is Ricci-flat.

\subsubsection*{Interface quantities $u$ and $v$}

Recall from Definition~\ref{def:uv} that
$u = (S_y - 2w)/(d_y w)$ and $v = (S_x - 2w)/(d_x w)$,
where $w = \mathbf{w}_{xy}$ is the Perron weight of the central edge.

Let $p$ (resp.\ $r$) denote the common weight of the leaf edges at $x$
(resp.\ $y$). Then $S_x = 2p + w$ and $S_y = 5r + w$, with $d_x = 3$,
$d_y = 6$.

From the eigenvalue equation at a leaf edge of $x$ with $\lambda_1 = 0$,
one obtains $w = 3p$, i.e.\ $p = w/3$. Hence
\[
  v = \frac{S_x - 2w}{d_x w}
    = \frac{2p + w - 2w}{3w}
    = \frac{2p - w}{3w}
    = \frac{2(w/3) - w}{3w}
    = -\frac{1}{9}.
\]

From the eigenvalue equation at a leaf edge of $y$ with $\lambda_1 = 0$,
one obtains $w = 3r$, i.e.\ $r = w/3$. Hence
\[
  u = \frac{S_y - 2w}{d_y w}
    = \frac{5r + w - 2w}{6w}
    = \frac{5r - w}{6w}
    = \frac{5(w/3) - w}{6w}
    = \frac{1}{9}.
\]

Thus
\[
  u = \frac{1}{9},\qquad v = -\frac{1}{9},\qquad
  uv = -\frac{1}{81},\qquad u + v = 0 = \lambda_1.
\]
\subsubsection*{One-step subdivision}

For the subdivided tree $T^{(2)}$, Theorem~\ref{thm:one-step} gives
\[
  \lambda_2 < \lambda_1 \;\Longleftrightarrow\; uv > -\frac{\lambda_1}{2}.
\]

Since $\lambda_1=0=\lambda_{\max}(R_T)$, Lemma~\ref{lem:lambda1-above-R0} gives
$-R_0\succ0$. Moreover, Proposition~\ref{prop:uv}(3) yields
\[
  \det S(0)=uv+\frac{\lambda_1}{2}=-\frac1{81}<0.
\]
Hence $S(0)$ is indefinite. By Lemma~\ref{lem:schur}(S2),
$-R_{T^{(2)}}$ is not positive semidefinite. Therefore
\[
  \lambda_2=\lambda_{\max}(R_{T^{(2)}})>0=\lambda_1.
\]

The exact value of $\lambda_2$ is the unique positive root of
$\Phi_A(\lambda)\Phi_B(\lambda)=1$. Using the explicit formulas
\[
  \Phi_A(\lambda)=2\lambda+\frac53-\frac{4}{9(\lambda+1)},\qquad
  \Phi_B(\lambda)=2\lambda+\frac43-\frac{5}{9(2\lambda+1)},
\]
numerical solution yields $\lambda_2 \approx 0.008599$. 

\subsubsection*{Feedback functions and long-chain limit}

The feedback functions follow from Proposition~\ref{prop:local-feedback}.
For the left branch ($d_x=3$, two leaves) and the right branch
($d_y=6$, five leaves),
\[
  \Phi_A(\lambda)=\beta_2(\lambda)=2\lambda+\frac53-\frac{4}{9(\lambda+1)},
  \qquad
  \Phi_B(\lambda)=\beta_5(\lambda)=2\lambda+\frac43-\frac{5}{9(2\lambda+1)},
\]
where $\beta_s$ is as in Proposition~\ref{prop:ss}. Recall
$q(\lambda)=1+\lambda-\sqrt{\lambda(\lambda+2)}$, the decaying mode of the
chain recurrence. At $\lambda=0$,
\[
  \Phi_A(0)=\frac{11}{9}>1=q(0),\qquad
  \Phi_B(0)=\frac{7}{9}<1=q(0).
\]
Since $\Phi_A$ is strictly increasing and $q$ strictly decreasing, the
equation $\Phi_A(\lambda)=q(\lambda)$ has \emph{no} positive root: the left
branch alone cannot support a positive limit. In contrast,
$\Phi_B(\lambda)=q(\lambda)$ has a unique positive root
\[
  \lambda_B\approx 0.0172371478.
\]

Direct numerical diagonalization of $R_{T^{(\ell)}}$ (validated by the
sanity checks $\lambda_{\max}(R_{T_{3,\ell}})\equiv0$ and
$\lambda_{4,\infty}\approx0.00716$) gives
\[
  \lambda_2\approx0.008599,\quad
  \lambda_3\approx0.0123,\quad
  \lambda_4\approx0.0142,\quad
  \lambda_5\approx0.0153,\ \ldots,\quad
  \lambda_\ell\uparrow\lambda_B\approx0.017237.
\]
The sequence is strictly \emph{increasing} from $\lambda_1=0$ and converges
to $\lambda_B$: the right branch, being the more ``responsive'' one
($\Phi_B(0)<q(0)$), governs the long-chain limit, while the ``stiff'' left
branch ($\Phi_A(0)>q(0)$) contributes no matching root.

Finally, the limit value is not merely numerical coincidence. Because the
right branch is one half of a $5$-star, $\Phi_B\equiv\beta_5$ \emph{identically},
so the limit equation $\Phi_B(\lambda)=q(\lambda)$ is literally the same as the
long-chain limit equation $\beta_5(\lambda)=q(\lambda)$ of the \emph{symmetric}
double star $S_{5,5}$. Thus $\lambda_B$ coincides exactly with
$\lambda_{5,\infty}$; the convergence $\lambda_\ell\to\lambda_B$ itself is the
asymmetric instance of the conjecture in Remark~\ref{rem:convergence}(2),
left for future work.

\begin{figure}[ht]
\centering
\begin{tikzpicture}
\begin{axis}[
  width=0.72\textwidth, height=0.3\textwidth,
  xlabel={subdivision length $\ell$},
  ylabel={$\lambda_\ell=\lambda_{\max}(R_{T^{(\ell)}})$},
  xmin=1.5, xmax=42, ymin=-0.003, ymax=0.05,
  xtick={2,5,10,20,30,40},
  legend pos=north east, legend cell align=left,
  grid=both, grid style={gray!18},
  tick label style={font=\small}, label style={font=\small},
  every axis plot/.append style={thick}]

% S_{4,4}: decreasing to positive limit
\addplot[mark=*,blue] coordinates {
 (2,0.04721360)(3,0.03503235)(4,0.02811377)(5,0.02365655)
 (6,0.02055044)(8,0.01652015)(12,0.01236967)(20,0.00914021)(40,0.00736850)};
\addlegendentry{$S_{4,4}$ (decreasing)}
\addplot[dashed,blue,forget plot] coordinates {(1.5,0.00716036)(42,0.00716036)};
\node[blue,font=\scriptsize,anchor=west] at (axis cs:30,0.0102)
   {$\lambda_{4,\infty}\!\approx\!0.00716$};

% asymmetric (3,6): increasing to lambda_B
\addplot[mark=triangle*,red] coordinates {
 (2,0.00859922)(3,0.01228730)(4,0.01420520)(5,0.01530685)
 (6,0.01597788)(8,0.01667726)(12,0.01711696)(20,0.01723108)(40,0.01723714)};
\addlegendentry{$(d_x,d_y)=(3,6)$ (increasing)}
\addplot[dashed,red,forget plot] coordinates {(1.5,0.01723715)(42,0.01723715)};
\node[red,font=\scriptsize,anchor=west] at (axis cs:24,0.0192)
   {$\lambda_B=\lambda_{5,\infty}\!\approx\!0.01724$};

% S_{3,3}: constant zero
\addplot[mark=square*,black] coordinates {
 (2,0)(3,0)(4,0)(5,0)(6,0)(8,0)(12,0)(20,0)(40,0)};
\addlegendentry{$S_{3,3}$ (constant $\equiv0$)}

% lambda_1 = 0 reference for asymmetric start
\node[font=\scriptsize,anchor=south west] at (axis cs:2,0.0005){$\lambda_1=0$};
\end{axis}
\end{tikzpicture}
\caption{The three regimes of edge subdivision for the Ricci matrix, in sharp
contrast to the unconditional decrease of the adjacency spectral radius
(Hoffman--Smith). The symmetric double star $S_{4,4}$ \emph{decreases} to a
positive limit; the threshold family $S_{3,3}$ stays Ricci-flat
($\lambda_\ell\equiv0$) for every $\ell$; and the asymmetric tree with
$(d_x,d_y)=(3,6)$ \emph{increases} from $\lambda_1=0$ to
$\lambda_B=\lambda_{5,\infty}$, the limit being governed by the more responsive
($5$-leaf) branch. Dashed lines mark the long-chain limits.}
\label{fig:three-regimes}
\end{figure}
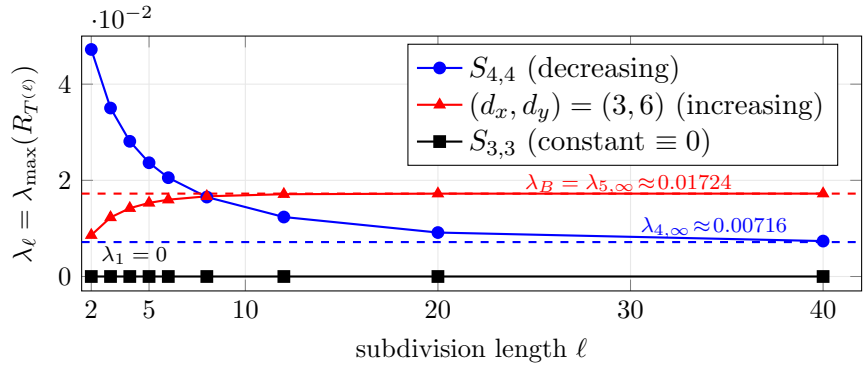

\noindent\textbf{Acknowledgements.}
S.~Bai is supported by NSFC, no.~12301434. B.~Hua is supported by NSFC, no.~12371056.

\bibliographystyle{plain}
\bibliography{referencesxifen}

\end{document}